\documentclass[a4paper,11pt]{article}
\usepackage{enumerate}
\usepackage{amsmath,amsthm,amssymb}
\usepackage{tikz,hyperref}
\usepackage[mathlines]{lineno}

\parskip \bigskipamount

\title{Decomposing the Complete $r$-Graph}
\author{Imre Leader\thanks{Department of Pure Mathematics and Mathematical Statistics, Centre for Mathematical Sciences, University of Cambridge, Wilberforce Road, Cambridge CB3 0WB, United Kingdom. Email: \texttt{I.Leader@dpmms.cam.ac.uk}.} \and Luka Mili\'cevi\'c\thanks{Department of Pure Mathematics and Mathematical Statistics, Centre for Mathematical Sciences, University of Cambridge, Wilberforce Road, Cambridge CB3 0WB, United Kingdom. Email: \texttt{lm497@cam.ac.uk}.} \and Ta Sheng Tan\thanks{Institute of Mathematical Sciences, Faculty of Science, University of Malaya, 50603 Kuala Lumpur, Malaysia. Email: \texttt{tstan@um.edu.my}. This author acknowledges support received from the Ministry of Higher Education of Malaysia via FRGS grant FP048-2014B.}}

\newtheorem{thm}{Theorem}
\newtheorem{proposition}[thm]{Proposition}

\newtheorem{conjecture}[thm]{Conjecture}
\newtheorem{question}[thm]{Question}
\theoremstyle{remark}  

\theoremstyle{definition}

\begin{document}

\maketitle
\begin{abstract}
 Let $f_r(n)$ be the minimum number of complete $r$-partite $r$-graphs needed to partition the edge set of the complete $r$-uniform hypergraph on $n$ vertices. Graham and Pollak showed that $f_2(n) = n-1$.
 An easy construction shows that $f_r(n)\le (1-o(1))\binom{n}{\lfloor r/2\rfloor}$ and it has been unknown if this upper bound is asymptotically sharp.
 In this paper we show that $f_r(n)\le (\frac{14}{15}+o(1))\binom{n}{r/2}$ for each even $r\ge 4$.
\end{abstract}

\emph{Keywords: }Hypergraph, Decomposition, Graham-Pollak

\section{Introduction}

The edge set of $K_n$, the complete graph on $n$ vertices, can be partitioned into $n-1$ complete bipartite subgraphs: this may be done in many ways, for example by taking $n-1$ stars centred at different vertices. Graham and Pollak~\cite{graham1,graham2} proved that the number $n-1$ cannot be decreased. Several other proofs of this result have been found, by Tverberg~\cite{tverberg}, Peck~\cite{peck}, and Vishwanathan~\cite{vishwanathan1,vishwanathan2}.

Generalising this to hypergraphs, for $n\ge r\ge 1$, let $f_r(n)$ be the minimum number of complete $r$-partite $r$-graphs needed to partition the edge set of $K_n^{(r)}$, the complete $r$-uniform hypergraph on $n$ vertices
(i.e., the collection of all $r$-sets from an $n$-set).
Thus the Graham-Pollak theorem asserts that $f_2(n) = n-1$. For $r\ge 3$, an easy upper bound of $\binom{n-\lceil r/2\rceil}{\lfloor r/2\rfloor}$ may be obtained by generalising the star example above. 
Indeed, having ordered the vertices, consider the collection of $r$-sets whose $2nd, 4th, \ldots, (2\lfloor r/2\rfloor)th$ vertices are fixed. This forms a complete $r$-partite $r$-graph, and the collection of all $\binom{n-\lceil r/2\rceil}{\lfloor r/2\rfloor}$ such is a partition of $K_n^{(r)}$. 
(There are many other constructions achieving the exact same value - see, for example Alon's recursive construction in \cite{alon1}.)

Alon~\cite{alon1} showed that $f_3(n) = n-2$. More generally, for each fixed $r\ge 1$, he showed that  
$$\frac{2}{\binom{2\lfloor r/2 \rfloor}{\lfloor r/2 \rfloor}}(1+o(1))\binom{n}{\lfloor r/2\rfloor}\le f_r(n)\le (1-o(1))\binom{n}{\lfloor r/2\rfloor},$$
where the upper bound is from the construction above.

The best known lower bound for $f_r(n)$ was obtained by Cioab\v{a}, K\"{u}ngden and Verstra\"{e}te~\cite{cioaba1}, who showed that $f_{2k}(n) \ge \frac{2\binom{n-1}{k}}{\binom{2k}{k}}.$
For upper bounds for $f_r(n)$, the above construction is not sharp in general. 
Cioab\v{a} and Tait~\cite{cioaba2} showed that $f_6(8) = 9<\binom{8-3}{3}$,  and used this to give an improvement in a lower-order term, showing that $f_{2k}(n) \le \binom{n-k}{k} - 2\left\lfloor\frac{n}{16}\right\rfloor \binom{\lfloor\frac{n}{2}\rfloor-k+3}{k-3}$ for any $k\ge 3$. 
(We mention briefly that any improvement of $f_4(n)$ for any $n$ will further improve the above upper bound. Indeed, one can check that $f_4(7) = 9< \binom{7-2}{2}$, and this will imply that $f_r(n)\le \binom{n-\lfloor r/2\rfloor}{\lfloor r/2\rfloor} - cn^{\lfloor r/2\rfloor-1}$ for some positive constant $c$. But note that, again, this is only an improvement to a lower-order term.)

Despite these improvements, the asymptotic bounds of Alon have not been improved. 
Perhaps the most interesting question was whether the asymptotic upper bound is the correct estimate.

Our aim of the paper is to show that the asymptotic upper bound is not correct for each even $r\ge 4$. In particular, we will show that $$f_4(n) \le \frac{14}{15}(1+o(1))\binom{n}{2},$$
and obtain the same improvement of $\frac{14}{15}$ for each even $r\ge 4$.

A key to our approach will be to consider a related question:
what is the minimum number of products of complete bipartite graphs, that is, sets of the form $E(K_{a,b})\times E(K_{c,d})$, needed to partition $E(K_n)\times E(K_n)$?
There is an obvious guess, namely that we take the product of the complete bipartite graphs in the partitions of both $K_n$s.
This gives a partition using $(n-1)^2$ products of complete bipartite graphs.
But can we improve this?
Writing $g(n)$ for the minimum value, it will turn out that, unlike for $f_4$, any improvement in the value of $g(n)$ for one $n$ gives an asymptotic improvement for $g$ as well.
In this sense, this means that $g$ is a `better' function to investigate than $f_4$. 

The plan of the paper is as follows. 
In Section~\ref{sec2}, we show how the function $g$ is related to $f_4$, and give some related discussions. 
Then in Section~\ref{sec3}, we investigate the simplest product of complete graphs: we attempt to partition the product set $E(K_3)\times E(K_n)$ into products of complete bipartite graphs. 
Although Section~\ref{sec3} is not strictly needed for our final bounds, it does provide several ideas and motivation for later.
In Section~\ref{sec4}, we prove our main result on $g$ and from this on $f_4$. 
Finally, in Section~\ref{sec5}, we mention some remarks and open problems. 

We use standard graph and hypergraph language throughout the paper. 
For an $r$-uniform hypergraph $H$, let $f_r(H)$ be the minimum number of complete $r$-partite $r$-uniform hypergraphs needed to partition the edge set of $H$. So $f_r\left(K_n^{(r)}\right)$ is just $f_r(n)$.
A \emph{minimal decomposition} of an $r$-graph $H$ is a partition of the edge set of $H$ into $f_r(H)$ complete $r$-partite $r$-graphs.
A \emph{block} is a product of the edge sets of two complete bipartite graphs. 
For graphs $G$ and $H$, let $g(G, H)$ be the minimum number of blocks needed to partition the set $E(G)\times E(H)$. Thus $g(n) = g(K_n,K_n)$.
Similarly, a \emph{minimal decomposition} of $E(G)\times E(H)$ is a partition of the set into $g(G,H)$ blocks.

\section{Products of complete bipartite graphs}\label{sec2}

We start by showing how $g$ is related to $f_4$.

\begin{proposition}\label{prop_fn}
 Let $\alpha>0$ be a constant. If $g(n)\le \alpha n^2$ for all $n$, then $f_4(n)\le \alpha(1+o(1))\frac{n^2}{2}$.
\end{proposition}

 \begin{proof}
  We will show that 
  \begin{equation}\label{eqn_fn}
   f_4(n)\le \alpha\left(\frac{n^2}{2}\right) + Cn\log n
  \end{equation}
  for some sufficiently large $C$.
  This is clearly true for $n\le 4$. So assume $n>4$ and the inequality \eqref{eqn_fn} holds for $1,2,\ldots,n-1$. 
  We will consider the case when $n$ is even - the case when $n$ is odd is similar.
  
  In order to decompose the edge set of $K_n^{(4)}$, we can split the $n$ vertices into two equal parts, say $V\left(K_n^{(4)}\right) = A\cup B$, where $|A|=|B|= n/2$. The sets of 4-edges $\{e: e\subset A$\} and $\{e:e\subset B$\} can each be decomposed into $f_4(n/2)$ complete 4-partite 4-graphs; the sets of 4-edges $\{e: |e\cap A|=3\}$ and $\{e: |e\cap B|=3\}$ can each be decomposed into $f_3(n/2)$ complete 4-partite 4-graphs; while the remaining set of 4-edges $\{e: |e\cap A| = |e\cap B|=2\}$ can be decomposed into $g(n/2)$ complete 4-partite 4-graphs. So by the assumption of $g(n)$ and the induction hypothesis, we have
  \begin{align*}
   f_4(n) &\le 2f_4(n/2) + g(n/2) + 2f_3(n/2)\\
	  &\le 2\left(\alpha\left(\frac{n^2}{8}\right)+\frac{Cn}{2}\log\left(\frac{n}{2}\right)\right) + \alpha\left(\frac{n}{2}\right)^2 + 2\left(\frac{n}{2}-2\right)\\
	  &\le \alpha\left(\frac{n^2}{2}\right) + Cn\log n.
  \end{align*}
 \end{proof}
 
 In the Introduction, we mentioned that any improvement in the upper bound of $f_4(n)$ from the easy upper bound of $\binom{n-2}{2}$, for one fixed $n$, will lead to an improvement for all (greater) values of $n$, but not an asymptotic improvement.
 However, very helpfully, this is not the case for $g$. 
 Indeed, any improvement to $g(n)$ for one particular $n$ leads to an asymptotic improvement.
 This is the content of the following simple proposition.
 
 \begin{proposition}\label{prop_gn}
  Suppose $g(K_{a},K_{b}) < (a-1)(b-1)$ for some $a$ and $b$. Then $g(n) \le \beta n^2$ for all $n$, for some constant $\beta<1$.
 \end{proposition}
 
 \begin{proof}
 Suppose $g(K_{a},K_{b}) = c<(a-1)(b-1)$ for some fixed $a$ and $b$. Then, setting $\alpha = \frac{c}{(a-1)(b-1)}$, we will show that $$g(K_{1+(a-1)i}, K_{1+(b-1)j})\le \alpha((a-1)i)((b-1)j)=cij$$ for any $i,j\ge 1$. This will then imply that $g(1+(a-1)(b-1)k) \le \alpha((a-1)(b-1)k)^2$ for any $k\ge 1$, and hence
  $$g(n) \le \alpha n^2 + Cn $$ for some constant $C$.
 
  We proceed by induction on $i$. 
  We will show the base case of $g(K_{a}, K_{1+(b-1)j})\le cj$ by induction on $j$. The case $j=1$ is true by assumption. So fix $j>1$ and by induction, we have $g(K_a, K_{1+(b-1)(j-1)})\le c(j-1)$.
  
  Let $G=K_{b}$ be a subgraph of $K_{1+(b-1)j}$.
  Note that $K_{1+(b-1)j} - G$ (i.e., the graph $K_{1+(b-1)j}$ with the edges of $G$ removed) is a blow-up of $K_{1+(b-1)(j-1)}$ by replacing one of its vertices with an empty graph on $b$ vertices. So $g(K_{a}, K_{1+(b-1)j} - G) = g(K_a, K_{1+(b-1)(j-1)})\le c(j-1)$, implying
  \begin{align*}
   g(K_{a}, K_{1+(b-1)j}) &\le g(K_{a}, G) + g(K_{a}, (K_{1+(b-1)j} - G))\\
   &\le g(K_{a}, K_{b}) + c(j-1)\\ 
   &\le cj.
  \end{align*}
  
  Now fix $i>1$ and assume the theorem is true for $i-1$. That is, 
  $$g(K_{1+(a-1)(i-1)}, K_{1+(b-1)j})\le c(i-1)j$$ for all $j\ge 1$.
  To decompose $E(K_{1+(a-1)i})\times E(K_{1+(b-1)j})$ for any fixed $j$, we first let $H = K_{a}$ and note that $K_{1+(a-1)i} - H$ is a blow-up of $K_{1+(a-1)(i-1)}$ by replacing one of its vertices with an empty graph on $a$ vertices. Therefore,
  \begin{align*}
   g(K_{1+(a-1)i}, K_{1+(b-1)j}) &\le g(H, K_{1+(b-1)j}) + g((K_{1+(a-1)i} - H), K_{1+(b-1)j})\\
   &\le g(K_{a}, K_{1+(b-1)j}) + g(K_{1+(a-1)(i-1)}, K_{1+(b-1)j})\\ 
   &\le cj + c(i-1)j \qquad \mbox{(by the base case and induction hypothesis)}\\
   &= cij.
  \end{align*}
  This completes the proof of the proposition.
 \end{proof}

 From Proposition~\ref{prop_fn} and Proposition~\ref{prop_gn}, in order to improve the asymptotic upper bound on $f_4(n)$, it is enough to find $a$ and $b$ such that $g(K_a,K_b) < (a-1)(b-1)$. 
 
 The rest of this section is a digression (and so could be omitted if the reader wishes). 
 The question of whether or not $g(n)=(n-1)^2$ has the flavour of a `product' question. 
 Indeed, it is an example of the following general question.
 Suppose we have a set $X$ and a family $\mathcal{F}$ of some subsets of $X$, and we write $c(X,\mathcal{F})$ for the minimum number of sets in $\mathcal{F}$ needed to partition $X$.
 Is it true that $c(X\times Y,\mathcal{F}\times \mathcal{G}) = c(X,\mathcal{F})c(Y,\mathcal{G})$, where $\mathcal{F}\times \mathcal{G}=\{F\times G: F\in \mathcal{F}, G\in \mathcal{G}\}$?
 
 This is certainly not \emph{always} true. Indeed, for a simple example, let $X=\{1,2,\ldots,7\}$ and $\mathcal{F}=\{A\subset X: |A| = 1\text{ or }4\}$. 
 Clearly, $c(X,\mathcal{F})=4$.
 But $X\times X$ can be partitioned into four 3 by 4 rectangles and a single point,  
 giving $c(X\times X,\mathcal{F}\times \mathcal{F}) \le 13$. 
 
 However, there are a few cases where such a product theorem is known. 
 For example, Alon, Bohman, Holzman, and Kleitman~\cite{alon2} proved that if $X$ is a finite set of size at least 2, then any partition of $X^n$ into proper boxes must consist of at least $2^n$ boxes.
 Here, a \emph{box} is a subset of $X^n$ of the form $B_1\times B_2\times \ldots \times B_n$, where each $B_i$ is a subset of $X$.
 A box is \emph{proper} if $B_i$ is a proper subset of $X$ for every $i$.
 Note that this corresponds to a product theorem where $\mathcal{F}$ is the family of all proper subsets of $X$.
 (There are also some related results by Ahlswede and Cai in \cite{ahlswede1,ahlswede2}.)
 
 Unfortunately, we have not been able to prove any product theorem that might relate to our problem about $g(n)$.
 Indeed, it seems difficult to extend the result of Alon, Bohman, Holzman, and Kleitman at all. 
 For example, here are two closely related problems that we cannot solve.

 A box is \emph{odd} if its size is odd. Let $X$ be a finite set such that $|X|$ is odd. We can partition $X^n$ into $3^n$ odd proper boxes - can we do better?
 
 \begin{question}
  Let $X$ be a finite set such that $|X|$ is odd. Must any partition of $X^n$ into odd proper boxes consist of at least $3^n$ boxes?
 \end{question}
 
 We do not even see how to answer this question when $|X|=5$.

 A collection of proper boxes $B^{(1)},B^{(2)},\ldots, B^{(m)}$ of $X^n$ is said to form a \emph{uniform cover} of $X^n$ if every point of $X^n$ is covered the same number of times.
 
 \begin{question}
  Let $X$ be such that $|X|\ge 2$.
  Suppose $B^{(1)},B^{(2)},\ldots, B^{(m)}$ forms a uniform cover of $X^n$.
  Must we have $m\ge 2^n$?
 \end{question}

\section{Decomposing \texorpdfstring{$E(K_3)\times E(K_n)$}{K3}}\label{sec3}

In this section, we investigate $g(K_3,K_n)$.
As we know, we can decompose $E(K_3)\times E(K_n)$ using $2(n-1)$ blocks, and the question is whether we can improve this.

It turns out that the Graham-Pollak theorem actually gives some restriction on how small $g(K_3,K_n)$ can be.
To be more precise, we will need a weighted version of the Graham-Pollak theorem.
For the sake of completeness, we will include a proof here, although we stress that this is just a rewriting of the usual proof of the Graham-Pollak theorem.

Given a graph $G$ and a real number $\alpha$, we write $\alpha\cdot G$ for the weighted graph where each edge of $G$ is given a weight of $\alpha$. 
A collection of subgraphs $G_1,G_2,\ldots,G_m$ of $K_n$ is a \emph{weighted decomposition of $K_n$} if there exists real numbers $\alpha_1,\alpha_2,\ldots,\alpha_m$ such that
for each edge $e$ of $K_n$ we have $\displaystyle \sum_{i:e\in G_i} \alpha_i\: =\: 1$.
Note that the $\alpha_i$ are allowed to be negative.

\begin{thm}\label{weighted_GP}
 The minimum number of complete bipartite graphs needed to form a weighted decomposition of $K_n$ is $n-1$.
\end{thm}
\begin{proof}
 Let the vertex set of $K_n$ be $V = \{1,2,\ldots,n\}$ and associate each vertex $i$ with a real variable $x_i$. 
 Let $G$ be a complete bipartite subgraph of $K_n$ with vertex classes $X$ and $Y$. 
 Then we can define $Q(G) = L(X)\cdot L(Y)$, where $L(A) = \sum_{i\in A} x_i$ for any subset $A\subset V$.
 
 Suppose the bipartite graphs $G_k, 1\le k\le q$ with vertex classes $X_k$ and $Y_k$ form a weighted decomposition of $K_n$. Then we must have 
 $$\sum_{i<j} x_ix_j = \sum_{k=1}^q \alpha_kL(X_k)L(Y_k)$$ for some real $\alpha_1,\alpha_2,\ldots, \alpha_q$. Rewriting the left-hand-side of the above equation, we have 
 \begin{align*}
  \left(\sum_{i=1}^n x_i\right)^2 - \sum_{i=1}^n x_i^2 &= 2\sum_{k=1}^q \alpha_kL(X_k)L(Y_k).
 \end{align*}
 
 It follows that the linear subspace of $\mathbb{R}^n$ determined by the $q+1$ linear equations $\sum_{i=1}^n x_i = 0$ and $L(X_i) = 0, 1\le i\le q$,
 must be the zero subspace. Hence $q+1\ge n.$
\end{proof}

\begin{proposition}\label{prop_k3kn}
 For $n\ge 2$ we have $$\frac{9}{5}(n-1)\le g(K_3, K_n)\le 2(n-1).$$
\end{proposition}
\begin{proof}
The upper bound has been explained already. 
For the lower bound,
suppose the blocks $H_1, H_2, \ldots, H_q$ form a decomposition of $E(K_3)\times E(K_n)$. Then for each edge $e\in E(K_n)$, restricting the decomposition to the subset $E(K_3)\times e$, one of the following happens: either the three elements of $E(K_3)\times e$ decompose into three different $H_i$, or else two of the sets are in the same $H_i$ for some $i$ and the third set is in $H_j$ for some $j\ne i$. 

Let $G_0$ be the subgraph of $K_n$ spanned by the set of $e$ such that the first of these happens, and $G_1,G_2,G_3$ be the subgraphs of $K_n$ spanned by the set of $e$ for each of the three possible ways for the second case to happen, respectively. Thus in total we have
\begin{equation}\label{eq1}
 q\ge f_2(G_1) + f_2(G_2) + f_2(G_3) + f_2(G_0\cup G_1) + f_2(G_0\cup G_2) + f_2(G_0\cup G_3).
\end{equation}

Now, since $G_0,G_1,G_2,G_2$ form a partition of the edge set of $K_n$, we must have 
\begin{equation}\label{eq2}
 f_2(G_i)+f_2(G_j)+f_2(G_0\cup G_k)\ge n-1  
\end{equation}
 for any $\{i,j,k\} = \{1,2,3\}.$
 Next, note that $1\cdot (G_0\cup G_i), 1\cdot (G_0\cup G_j), (-1)\cdot (G_0\cup G_k), 2\cdot G_k$ form a weighted decomposition of $K_n$ for any $\{i,j,k\} = \{1,2,3\}$, so by Theorem~\ref{weighted_GP}, we must have 
 \begin{equation}\label{eq3}
  f_2(G_0\cup G_1)+f_2(G_0\cup G_2)+f_2(G_0\cup G_3)+f_2(G_i)\ge n-1
 \end{equation}
 for any $i=1,2,3$.
 
 Let $x = \frac{1}{3}\left(f_2(G_1)+f_2(G_2)+f_2(G_3)\right)$ and $y = \frac{1}{3}\left(f_2(G_0\cup G_1)+f_2(G_0\cup G_2)+f_2(G_0\cup G_3)\right)$. Summing over different $\{i,j,k\}$ for inequality~\eqref{eq2}, we get $2x+y\ge n-1$; while summing over different $i$ for inequality~\eqref{eq3}, we get $x+3y\ge n-1$. This implies that $x+y\ge \frac{3}{5}(n-1)$, and together with inequality~\eqref{eq1}, we conclude that
 \begin{align*}
  q &\ge 3x + 3y,\\
  \mbox{ i.e. }\quad q &\ge \frac{9}{5}(n-1).
 \end{align*}
\end{proof}
Note that for any partition of $K_n$ into $G_0, G_1,G_2,G_3$, we do obtain that $g(K_3,K_n)$ is at most the right-hand-side of \eqref{eq1}.

We believe that the only restriction on $g(K_3,K_n)$ should be the restriction coming from the Graham-Pollak theorem, namely that $g(K_3,K_n)\ge \frac{9}{5}(n-1)$. 
However, we have been unable to find any decomposition of $E(K_3)\times E(K_n)$ into fewer than $2(n-1)$ blocks.

\begin{question}
 Does there exist a constant $\alpha<2$ such that $g(K_3, K_n)\le (\alpha+o(1))n$? In particular, can we take $\alpha=\frac{9}{5}$?
\end{question}

\section{Decomposing \texorpdfstring{$E(K_4)\times E(K_n)$}{K4}}\label{sec4}

 The aim of this section is to find some $a,b$ in which $E(K_a)\times E(K_b)$ can be partitioned into fewer than $(a-1)(b-1)$ blocks.
 In the previous section, we looked at decompositions of $E(K_3)\times E(K_n)$ by considering all the four possible ways to decompose $E(K_3)$ into complete bipartite graphs - this induced four subgraphs that partitioned the edge set of $K_n$. 
 
 Now, those decompositions of $K_3$ involved three `large' complete bipartite subgraphs (namely, the $K_{1,2}$s), which between them form a 2-cover of $K_3$ (each edge of $K_3$ is in exactly two of them). However, this is in a sense `wasteful', as by the Graham-Pollak theorem, we might expect to find a uniform cover by three `large' complete bipartite subgraphs of $K_4$, rather than $K_3$. 
 
 This suggests that we should look at $E(K_4)\times E(K_n)$ instead of $E(K_3)\times E(K_n)$.
 It also suggests that, in each $E(K_4)\times e$, we do not allow \emph{any} decomposition of $K_4$, but just four decompositions of $K_4$, three of which involves a `large' complete bipartite subgraph and the fourth of which consists of single edges. 
 More precisely, the three decompositions of $K_4$ we allow here each consists of a 4-cycle and two independent edges.
 The three pairs of independent edges from these decompositions in turn form another decomposition of $K_4$ (into six complete bipartite graphs, each of which is a single edge). 
  
 Let $C_1, C_2, C_3$ be the three different 4-cycles of $K_4$ and let $G_0, G_1, G_2, G_3$ be the subgraphs of $K_n$ (as in Proposition~\ref{prop_k3kn}) whose edge sets partition the edge set of $K_n$.
 Then the sets $E(C_1)\times E(G_1), E(C_2)\times E(G_2), E(C_3)\times E(G_3), E(K_4 - C_1)\times E(G_0\cup G_1), E(K_4 - C_2)\times E(G_0\cup G_2), E(K_4 - C_3)\times E(G_0\cup G_3)$ form a partition of $E(K_4)\times E(K_n)$. So $E(K_4)\times E(K_n)$ can be decomposed into 
 \begin{equation*}
   f_2(G_1)+f_2(G_2)+f_2(G_3)+2f_2(G_0\cup G_1)+2f_2(G_0\cup G_2)+2f_2(G_0\cup G_3)
 \end{equation*}
 blocks.
 
 By the same argument as in Proposition~\ref{prop_k3kn}, we have the following.
 
 \begin{proposition}
   For $n\ge 2$, we have $$\frac{12}{5}(n-1)\le g(K_4, K_n)\le 3(n-1).$$
 \end{proposition}

 Again, it seems plausible that the only constraint on $g(K_4,K_n)$ is the one coming from the Graham-Pollak theorem. 
  
\begin{conjecture}\label{conj_k4kn}
 $g(K_4,K_n) = \frac{12}{5}(1+o(1))n$.
\end{conjecture}

 While we are unable to resolve this conjecture, we \emph{are} able to find an example with $g(K_4,K_n)< 3(n-1)$. 
 We start by observing that $G_0\cup G_1, G_0\cup G_2, G_0\cup G_3$ form an \emph{odd cover} of $K_n$ (each edge of $K_n$ appears an odd number of times).
 Now, it is known (see, e.g., \cite{radhakrishnan}) that $K_8$ has an odd cover with four complete bipartite graphs. 
 Indeed, the four $K_{3,3}$s with vertex classes $V_1 = \{1,3,5\}\cup \{2,4,6\}, V_2 = \{1,4,7\}\cup \{2,3,8\}, V_3 = \{2,5,7\}\cup \{1,6,8\}$ and $V_4 = \{3,6,7\}\cup \{4,5,8\}$ respectively form an odd cover of $K_8$.
 If we break the symmetry by deleting two vertices (vertices 6 and 8) from this odd cover of $K_8$, we obtain an odd cover of $K_6$ by four complete bipartite graphs, two of which are now disjoint.
 The union of these two disjoint complete bipartite graphs, together with the other two complete bipartite graphs, will be our $G_0\cup G_1, G_0\cup G_2, G_0\cup G_3$.
 Remarkably, this does give rise to a decomposition of $E(K_4)\times E(K_6)$ into fewer than 15 blocks.
 
 \begin{proposition}\label{prop_k4k6}
  The set $E(K_4)\times E(K_6)$ can be decomposed into $14$ blocks. In other words, $g(K_4, K_6) \le 14<(4-1)(6-1)$.
 \end{proposition}
 \begin{proof}
  Let $G_0, G_1$, $G_2$, $G_3$ be graphs that form a decomposition of $K_6$, defined as follow:
  \begin{align*}
   &E(G_0) = \{12,34\},\\ 
   &E(G_0\cup G_1) = \big\{ij:i\in \{1,3,5\}, j\in \{2,4\} \big\}, \\
   &E(G_0\cup G_2) = \big\{ij:i\in \{1,4,6\}, j\in \{2,3\}\big\},\\
   &E(G_0\cup G_3) = \big\{ij:i\in\{3,6\}, j\in \{4,5\}\big\}\cup \{12,15,16\}.
  \end{align*}

  By construction, we have $f_2(G_0\cup G_1) = f_2(G_0\cup G_2) = 1$, and $f_2(G_0\cup G_3) = 2$, and a quick check shows that
  $f_2(G_1) = f_2(G_2) = f_2(G_3) = 2$. So from the discussion above we have
  $$g(K_4, K_6) \le \sum_{i=1}^3 \big(f_2(G_i)+2f_2(G_0\cup G_i)\big) = 14.$$
 \end{proof}
  
 Combining Proposition~\ref{prop_fn}, Proposition~\ref{prop_gn} and Proposition~\ref{prop_k4k6}, we obtain our main result.

  \begin{thm}\label{thm_n}
  $f_4(n)\le \frac{14}{15}(1+o(1))\binom{n}{2}$.
 \end{thm}
 
 \section{Remarks and open problems}\label{sec5}

 Proposition~\ref{prop_k4k6} (together with Proposition~\ref{prop_gn}) implies that $g(n)\le \frac{14}{15}(1+o(1))n^2$.
 We do not believe $\frac{14}{15}$ is the correct constant, but we are not able to improve it.
 What about a lower bound of $g(n)$?
 From Proposition~\ref{prop_fn}, we know that if $g(n) = \alpha n^2(1+o(1))$, then we have $f_4(n)\le \alpha(1+o(1))\binom{n}{2}$. 
 So we must have $\alpha \ge \frac{1}{3}$ from Alon's result on the lower bound of $f_4(n)$.
 
 Here, we are able to give a small improvement, namely $\alpha\ge \frac{1}{2}$.
 For this, we will need a result by Reznick, Tiwari, and West~\cite{reznick} on decomposing weak product graphs into bipartite graphs. The \emph{weak product} $G*H$ of two graphs $G$ and $H$ has vertex set $\{(u,v):u\in V(G), v\in V(H)\}$ with $(u_1,v_1)\sim(u_2,v_2)$ if and only if $u_1\sim u_2$ in $G$ and $v_1\sim v_2$ in $H$.
 
 \begin{thm}[\cite{reznick}]\label{thm_weakproduct}
  The minimum number of complete bipartite graphs needed to partition the edge set of $K_n * K_n$ is $(n-1)^2+1$.
 \end{thm}
 
 \begin{proposition}
  For $n\ge 2$, we have $g(n)\ge \left\lceil\frac{(n-1)^2+1}{2}\right\rceil$.
 \end{proposition}
 \begin{proof}
  Suppose we can decompose $E(K_n)\times E(K_n)$ into $q$ blocks. For each of such blocks (say the parts from the left $K_n$ are $X_1,X_2$ and the parts from the right $K_n$ are $Y_1,Y_2$), we construct two complete bipartite graphs $G_1$ and $G_2$ as follows. The vertex classes of $G_1$ are $\{(x,y): x\in X_1, y\in Y_1\}$ and $\{(x,y): x\in X_2, y\in Y_2\}$; while the vertex classes of $G_2$ are $\{(x,y): x\in X_1, y\in Y_2\}$ and $\{(x,y): x\in X_2, y\in Y_1\}$.
  
  Observe that these $2q$ complete bipartite graphs partition the edge set of the weak product $K_n * K_n$. So by Theorem~\ref{thm_weakproduct}, we must have $$q\ge \left\lceil\frac{(n-1)^2+1}{2}\right\rceil.$$
 \end{proof}
 
In general, for any fixed $k$, can we improve the upper bound of $(k-1)(n-1)$ on $g(K_k, K_n)$ in a manner similar to what we have considered for $k=3$ and $k=4$? 
It seems that perhaps there is no $K_k$ having a `better' allowed sets of decompositions than the four allowed decompositions of $K_4$ that we used in Section~\ref{sec4}.
If this is correct, perhaps $\frac{4}{5}$ is the right constant even for $g(n)$.

\begin{question}
 Is it true that $g(n) = \frac{4}{5}(1+o(1))n^2$?
\end{question}

Finally, let us turn our attention to the function $f_r$ for $r>4$.
For fixed $r\ge 1$, let $\alpha_r$ be the smallest $\alpha$ such that $f_r(n) \le \alpha(1+o(1)) \binom{n}{\lfloor r/2\rfloor}$.
Thus the initial construction gives $\alpha_r \le 1$ for all $r$, while
Theorem~\ref{thm_n} says that $\alpha_4\le \frac{14}{15}$.
This implies that $\alpha_{r}\le \frac{14}{15}$ for all even $r$.

\begin{thm}
 For each fixed $k\ge 2$, we have $$f_{2k}(n)\le \frac{14}{15}(1+o(1))\binom{n}{k}.$$
\end{thm}
\begin{proof}
 We use induction on $k$. By Theorem~\ref{thm_n}, the result is true for the base case $k=2$. For larger $k$, the result is an easy consequence of the following inequality:
 $$f_{2k+2}(n) \le f_{2k}(n-2) + f_{2k}(n-3) + \ldots + f_{2k}(2k).$$
 This inequality is obtained by ordering the $n$ vertices and observing that the set of $(2k+2)$-edges whose second vertex is $i$, for any fixed $i\in \{2,3,\ldots,n-2k\}$, may be  decomposed into $f_{2k}(n-i)$ complete $(2k+2)$-partite $(2k+2)$-graphs.
\end{proof}

We do not see how to obtain a bound below 1 for $\alpha_r$ for $r$ odd.
But actually we would expect the following to be true.

\begin{conjecture}\label{qs_cr}
 We have $\alpha_r\rightarrow 0$ as $r\rightarrow\infty$.
\end{conjecture}

To prove this, it would be sufficient to show that $\alpha_5<1$.
Indeed, suppose $f_5(n) \le (\alpha+o(1))\binom{n}{2}$ for some $\alpha<1$.
Let $r = 6k-1$ and order the $n$ vertices.
We can decompose the complete $r$-graph on $n$ vertices by considering the set of $r$-edges whose $6$th, $12$th, $\ldots, 6(k-1)$th are $i_1,i_2,\ldots,i_{k-1}$ respectively, where $i_1\ge 6$ and $i_{k-1}\le n-5$ and $i_j-i_{j-1}\ge 6$ for $2\le j\le k-1$. 
For each such fixed $i_1,i_2,\ldots,i_{k-1}$, these $r$-edges can be decomposed into $f_5(i_1-1)f_5(i_2-i_1-1)\ldots f_5(i_{k-1}-i_{k-2}-1)f_5(n-i_{k-1})$
complete $r$-partite $r$-graphs.
Summing over all possible choices of $i_1,i_2,\ldots,i_{k-1}$, we deduce that $f_{6k-1}(n)\le (\alpha^k + o(1))\binom{n}{3k-1}$.

Annoyingly, we do not see how to use any of our arguments about $f_4$ for $f_5$. 

\begin{question}
 Is $\alpha_5<1$? In other words, do we have $f_5(n) \le (\alpha+o(1))\binom{n}{2}$ for some $\alpha<1$?
\end{question}

\end{document}